\documentclass{article}
\usepackage{tikz}
\usepackage{amssymb}
\usepackage{amsmath}
\usepackage{amsthm}
\usepackage{comment}
\newcommand{\luimw}{\textsc{lu-mim width}}
\newtheorem{cclaim}{Claim}

\newtheorem{definition}{Definition}
\newtheorem{theorem}{Theorem}
\newtheorem{lemma}{Lemma}
\newtheorem{example}{Example}
\newtheorem{proposition}{Proposition}
\date{}
\title{Classification of OBDD size for monotone 2-CNFs}
\author{Igor Razgon\\ Department of Computer Science, Birkbeck University of London \\ igor@dcs.bbk.ac.uk}

\begin{document}
\maketitle
\begin{abstract}
We introduce a new graph parameter called linear upper maximum induced
matching width \textsc{lu-mim width}, denoted for a graph $G$ by $lu(G)$.
We prove that the smallest size of the \textsc{obdd} for $\varphi$,
the monotone 2-\textsc{cnf} corresponding to $G$, is sandwiched 
between $2^{lu(G)}$ and $n^{O(lu(G))}$. 
The upper bound is based on a combinatorial statement that might 
be of an independent interest.
We show that the bounds in terms of this parameter are best possible.

\end{abstract}

\section{Introduction}
{\bf Statement of the results.}
Monotone 2-\textsc{cnf}s are \textsc{cnf}s with two positive literals
per clause. They can be viewed as graphs without isolated vertices.
In particular, for such a graph $G$, $\varphi=\varphi(G)$ is a \textsc{cnf}
consisting of clauses $(u \vee v)$ for each $\{u,v\} \in E(G)$. We refer to $G$
as the \emph{underlying graph} of $\varphi$.

In this paper we introduce a new graph parameter called
(Linear Upper Maximum Induced Matching Width) (\textsc{lu-mim width}).
This parameter is located 'in-between' of two existing parameters:
Linear Maximum Induced Matching Width (\textsc{lmim width}) \cite{VaThesis} 
and Linear Symmetric Induced Matching Width (\textsc{lsim width}) \cite{simwidthpaper}.
We prove that \textsc{lu-mim width} captures the size of Ordered Binary Decision Diagrams (\textsc{obdd}s) for
monotone 2-\textsc{cnf}s with a quasipolynomial gap.
In particular, we show that $2^{lu(G)} \leq obdd(\varphi) \leq n^{O(lu(G))}$
where $obdd(\varphi)$ is the smallest number of nodes of an \textsc{obdd}
for a monotone 2-\textsc{cnf} $\varphi$ and $lu(G)$ is the \textsc{lu-mim width}
of the underlying graph $G$ of $\varphi$. The upper bound is based on a 
combinatorial statement that may be of an
independent interest. In particular, we exhibit a connection of 
this statement to the Sauer-Shelah lemma (e.g. Theorem 10.1 in \cite{ExtrComb}). 

We show that the bounds are best possible by demonstrating classes
of graphs $G_1$ and $G_2$ such that 
$obdd(\varphi(G_1)) \geq n^{\Omega(lu(G_1))}$ and 
$obdd(\varphi(G_2)) \leq 2^{O(lu(G_2))}$. 

{\bf Motivation.}
Monotone \textsc{cnf}s are essentially hypergraphs while monotone $2$-\textsc{cnf}s are essentially graphs.
Therefore, it is natural to try to characterize the size of models of the corresponding 
Boolean functions by graph parameters. It is particularly neat if such a parameter can 'capture'
the size of a model on a class of \textsc{cnf}s, that is to tightly characterize
both upper and lower bounds. It is also desirable for the parameter to be well known as, in this case,  
existing techniques can be harnessed for determining the value of the parameter for the given
class of graphs. 

An example of such a neat capturing is characterization of the size of non-deterministic
read-once branching programs ($1$-\textsc{nbp}s) representing monotone $2$-\textsc{cnf}s $\varphi(G)$ where $G$
has a bounded degree. In this case, considering the degree constant, 
the size of the smallest $1$-\textsc{nbp} representing $\varphi(G)$ is $2^{\Theta(pw(G))}$ where $pw(G)$ is the pathwidth
of $G$: the upper bound has been established in \cite{VardiTWD},
the lower bound in terms of maximum matching width in \cite{RazgonAlgo} 
and it has been shown in \cite{obddtcompsys} that the maximum matching width and pathwdith are
linearly related. It is thus natural to ask whether such a capturing is possible for graphs of unbounded
degree. 

In this paper we address the above question \emph{partially}.
First, we obtain the result for \textsc{obdd}s, a special case of $1$-\textsc{nbp}s.
Generalization to $1$-\textsc{nbp}s is left as an open question. It is important to remark
that although, for bounded degrees, the pathwdith captures the sizes of both
models, for the case of unbounded degree another parameter might be needed for $1$-\textsc{nbp}s. 
Second, there is a quasipolynomial gap between the upper and lower bounds.
We believe that this is still reasonable because the value of the parameter provides a good 
indication of the size of the resulting \textsc{obdd}. Besides, we show that for the considered
parameter, no tighter capturing is possible. Third, we introduced a new parameter rather
than using an existing one. However, this parameter is closely related to existing ones.
In Section \ref{sec:whynew} we discuss why we cannot use the existing parameters
for the stated purpose.  

An additional motivation of the proposed results is that they contribute to understanding the
combinatorics of \textsc{mim width}, a parameter becoming increasingly popular
among graph algorithms researchers (see the related work part for the relevant references). 



{\bf Related work.}
Here we overview related results that have not been mentioned in the earlier
parts of the introduction. 
 
The size of Decomposable Negation Normal Forms (\textsc{ddnf}s) for monotone 2-\textsc{cnf}s
of bounded degree is captured by treewidth. 
In particular an \textsc{fpt} upper bound for \textsc{cnf}s of bounded (primal) treewidth is proved in Theorem 16 of \cite{DarwicheJACM}.
A matching lower bound for \textsc{cnf}s of bounded arity and bounded number of variable occurrences
follows from the combination of Theorem 8.3 and Lemma 8.4. \cite{AmarilliTCS} \footnote{I would like to thank
Florent Capelli for pointing me out to this result. }

A lower bound for \textsc{obdd}s for monotone \textsc{cnf}s is established in 
\cite{AmarilliTCS}. For $2$-\textsc{cnf}s, the lower bound is $2^{\Omega(pw(G)/d^2)}$
where $pw(G)$ and $d$ are the pathwdith and the max-degree of $G$.
The lower bound provided in this paper is better because 
$lu(G)=\Omega(pw(G)/d)$ due to pathwidth and linear maximum matching width
being linearly related \cite{obddtcompsys}. The proof of the $n^{O(lu(G))}$ upper bound is similar in spirit 
to the main combinatorial lemma of \cite{pswidth}. 

The \textsc{mim-width}  \cite{VaThesis} has 
proven useful for design of efficient algorithms for intractable problems
for restricted classes of graphs, see for example the recent series
of papers \cite{mimwidth1}, \cite{mimwidth2},\cite{mimwidth3}.
Lower bounds of \textsc{mim-width}  for several graph classes have been established in \cite{MengelDAM}. 


{\bf Structure of the paper.}
Section \ref{sec:prelim} introduces the necessary background.
Section \ref{sec:newparam} introduces the \textsc{lu-mim width} parameter.
Section \ref{sec:bounds} proves upper and lower bounds on the \textsc{obdd} size. 
Section \ref{sec:bestposs} proves that, in terms of the parameter, the bounds
are essentially tight. In Section \ref{sec:whynew} 
we discuss why we cannot use existing parameters for the purpose of
capturing \textsc{obdd} bounds for monotone $2$-\textsc{cnf}s.
Finally, Section \ref{sec:futres} outlines directions of further research.

\section{Preliminaries}  \label{sec:prelim}
A \emph{literal} is a Boolean variable or its negation.
We consider only \emph{proper} sets $S$ of literals where 
a variable cannot occur along with its negation.
The set of variables occurring in $S$
is denoted by $Var(S)$. 
A variable $x \in Var(S)$ can occur in $S$ either \emph{positively}, if $x \in S$
or \emph{negatively}, if $\neg x \in S$. We can also call $S$ an assignment (to $Var(S)$
if the clarification is needed). 

We view a \emph{Conjunctive Normal Form} (\textsc{cnf}) as a set of
\emph{clauses} and each clause is just a proper set of literals.
An assignment $C$ \emph{satisfies} a clause $C$ if $S \cap C \neq \emptyset$.
An assignment \emph{satisfies} a \textsc{cnf} $\varphi$ if $S \cap C \neq \emptyset$ for each
$C \in \varphi$.  For an assignment $S$, the \textsc{cnf} $\varphi|_S$ is obtained
from $\varphi$ by removal of all the clauses satisfied by $S$ and removal the
occurrences of $Var(S)$ from each remaining clause. We denote by $Var(\varphi)$
the set of all variables occurring in the clauses of $\varphi$.
Customarily, $|Var(\varphi)|$ is denoted by $n$. 
 
For a \textsc{cnf} $\varphi$, $U \subseteq Var(\varphi)$, let ${\bf A}(U)={\bf A}_{\varphi}(U)$ 
be the set of all assignments $S$ to $U$ that can be extended to
a satisfying assignment of $\varphi$.
We denote by ${\bf BF}(U)={\bf BF}_{\varphi}(U)$ the set of all Boolean functions 
represented by $\varphi|_A$ for $A \in {\bf A}(U)$.

\begin{example} \label{exam1}
Let $\varphi=(x_1 \vee x_2) \wedge (x_1 \vee x_3) \wedge (x_3 \vee x_4) \wedge (x_2 \vee x_4)$.
Let $U=\{x_1,x_2\}$. Then ${\bf A}(U)=\{\{x_1,x_2\},\{x_1,\neg x_2\}, \{\neg x_1, x_2\}$.
Note that $\{\neg x_1,\neg x_2\}$ is not included in ${\bf A}(U)$ because the assignment 
falsifies the clause $(x_1 \vee x_2)$ and hence cannot be extended to a satisfying assignment
of $\varphi$. 
Then ${\bf BF}(U)$  is the set of functions on $x_3,x_4$ represented by the following set of \textsc{cnf}s
$\{(x_3 \vee x_4), (x_3),(x_4)\}$.   
\end{example}

An \emph{Ordered Binary Decision Diagram} (\textsc{obdd}) is a popular model for representation
of Boolean functions. For the purpose of this paper, we do not need a formal
definition of \textsc{obbd}s because the only fact about \textsc{obdd}s we use is 
Proposition \ref{propobddbounds} but we provide a definition for the 
sake of completeness.

\begin{definition}
An \textsc{obdd} $Z$ is a directed acyclic graph (\textsc{dag}) with \emph{one source}
and \emph{two sinks}.
Each non-sink vertex has exactly two outgoing neighbours.
The vertices and edges of $Z$ are labelled in the way specified below. 

Each non sink vertex is labelled with a variable, one of the sinks is labelled
with $true$, the other is labelled with $false$. 
Let $u$ be a non-sink node of $Z$ labelled with a variable $x$.
Then one outgoing edge of $u$ is labelled with the positive literal of $x$,
that is $x$, the other is labelled with the negative literal of $x$, that is $\neg x$.

The labelling of non-sink nodes also needs to observe two principles: being \emph{read-once}
and being \emph{ordered}. 
The read-once property means that in any directed path $P$ of $Z$ the labels of all the non-sink
nodes of $P$ are distinct (no variable occurs twice).
Being ordered means that there is a permutation $\pi(Z)$ of the variables labelling the nodes
of $Z$ so that for any path $P$ from a non-sink node $u$ to a non-sink node $v$
the label of $u$ precedes in $\pi(Z)$ the label of $v$.

For a directed path $P$ of $Z$, we denote by $A(P)$ the set of literals labelling the edges 
of $P$. Let $x_1, \dots, x_n$  be the variables labelling the nodes of $Z$.  
 The function $f_Z$ represented by $Z$ is defined as follows.
 Let $S$ be a set of literals with $Var(S)=\{x_1, \dots, x_n\}$.
 Then $f_Z$ is $true$ on $S$ if and only if 
$Z$ has a path $P$ from the source to the $true$ sink such
that $A(P) \subseteq S$.  
\end{definition}

We refer the reader to \cite{Wegbook} for an extensive study of \textsc{obdd}s.
For the results of this paper,  we only need bounds on $obdd(\varphi)$, the smallest \textsc{obdd} size
for a \textsc{cnf}  $\varphi$  as in the next statement that follows from Theorem 3.1.4 of \cite{Wegbook}.

\begin{proposition} \label{propobddbounds}
\begin{enumerate}
\item Suppose that for each permutation $\pi$ of $Var(\varphi)$
there is a prefix $\pi'$ of $\pi$ such that $|{\bf BF}(\pi')| \geq m$. \footnote{Here
and in several other places we slightly abuse the notation by using a sequence
as a set. The correct use will always be clear from the context.}
Then $obdd(\varphi) \geq m$.
\item Assume that there is a permutation $\pi$ of $Var(\varphi)$ such that
for every prefix $\pi'$ of $\pi$, $|{\bf BF}(\pi')| \leq m$. 
Then $obdd(\varphi)=O(n*m)$. 
\end{enumerate}
\end{proposition} 

In case of \textsc{obdd}s representing monotone $2$-\textsc{cnf}s upper and lower bounds
can be stated in graph theoretical terms as described below.  
We follow a standard graph-theoretical notation.
In particular $G[U]$ denotes the subgraph induced by $U \subseteq V(G)$. 
$N(U)$ is the set of all neighbours of vertices of $U$ excluding $U$, the considered
graph may be added as a subscript if not clear from the context.
The \textsc{cnf} $\{(u\vee v)|\{u,v\}\in V(G)\}$ is denoted by $\varphi(G)$.

\begin{definition}
Let $U \subseteq V(G)$. We denote by ${\bf ISET}(U)$ the 
family of all the independent subsets
of $U$. Let $V=V(G) \setminus U$. We define ${\bf TRACES}(U)=\{N(S) \cap V| S \in {\bf ISET}(U)\}$.
The subscript $G$ can be used for ${\bf TRACES}(U)$ and ${\bf ISET}(U)$  if the graph in question
is not clear from the context. 
\end{definition}

\begin{example} \label{exam2}
Let $G$ be a graph with vertices $x_1,x_2,x_3,x_4$ 
and edges $\{x_1,x_2\}, \{x_1,x_3\}, \{x_2,x_4\}$, $\{x_3,x_4\}$. 
(This is the graph corresponding to the \textsc{cnf} considered in Example \ref{exam1}.)
Let $U=\{x_1,x_2\}$. Then ${\bf ISET}(U)=\{\emptyset, \{x_1\}, \{x_2\}\}$,
${\bf TRACES}(U)=\{\emptyset, \{x_3\},\{x_4\}\}$. 
\end{example}

Combination of Examples \ref{exam1} and \ref{exam2}
demonstrates that ${\bf TRACES}_G(U)$ and ${\bf B}_{\varphi}(U)$
are of the same size where $\varphi=\varphi(G)$.
The following lemma shows that this is not a coincidence. 

\begin{lemma} \label{uppertraces}
Let $\varphi=\varphi(G)$.
Then $|{\bf BF}(U)|=|{\bf TRACES}(U)|$. 
\end{lemma}

{\bf Proof.}
It is not hard to see that ${\bf A}(U)=\{A(S)|S \in {\bf ISET}(U)\}$ where $A(S)$ 
is an assignment on $U$ where all the elements of $S$ occur negatively 
and the rest occur positively.
Furthermore, it is not hard to see that $\varphi|_{A(S)}$ is a \textsc{cnf} of the form
$\{(u) |u \in N(S) \cap V\} \cup \{(u,v)|\{u,v\} \in E(G[V \setminus N(S)])\}$.
It follows that for $S_1,S_2 \in {\bf ISET}(U)$ and $N(S_1) \cap V=N(S_2) \cap V$,
$\varphi_{A(S_1)}=\varphi_{A(S_2)}$. Conversely, we need to show that 
if $N(S_1) \cap V$ and $N(S_2) \cap V$ are distinct then so are the functions of $\varphi|_{A(S_1)}$ and $\varphi|_{A(S_2)}$.
Assume w..l.o.g. the existence of $v \in (N(S_1) \cap V) \setminus (N(S_2) \cap V)$
This means that $v$ occurs positively in all satisfying assignments of $\varphi_{A(S_1)}$
but can occur negatively in $\varphi_{A(S_2)}$: just assign positively the rest of the variables. 
$\blacksquare$  

Finally, we need one more definition. 

\begin{definition}
Let $U,V \subseteq V(G)$. 
A $(U,V)$-matching is a matching of $G$ consisting of edges with one end in $U$  and the other in $V$.
Let $M$ be such a matching. We denote by $U(M)$ the set of ends of the edges of $M$ that belong
to $U$.  Let $S$ be an independent subset of $U$. We say that $S$ \emph{enables} an induced $(U,V)$ matching
if there is an induced $(U,V)$-matching $M$ with $U(M)=S$. 
\end{definition} 

\section{Linear upper induced matching width} \label{sec:newparam}
In this section we introduce the parameter of \emph{Linear Upper Maximum Induced 
Matching Width} (\textsc{lu-mim width}). In order to present the parameter in the right context
we compare it with two existing parameters: Linear Maximum Induced Matching
Width (\textsc{lmim width}) and Linear Symmetric Induced Matching Width
(\textsc{lsim width}). 

The definition of all three parameters follows the same pattern.
First, we fix a permutation $\pi=(v_1, \dots, v_n)$,
denote each $\{v_1, \dots, v_i\}$ by $V_i$ and define the width of
the prefix $(v_1, \dots, v_i)$ as the largest size of an induced
$(V_i,V(G) \setminus V_i)$-matching of some subgraph of $G$.   
The difference between the above three parameters is in the choice
of the subgraph. The rest of the definition is identical for all
the three parameters and also pretty standard: the width of $\pi$ is 
the largest width among all the prefixes of $\pi$
and the width of $G$ is the smallest width 
among all the permutations.





To define the width of a permutation prefix for \textsc{lu-mim width},
we need the notion of an \emph{upper subgraph} introduced in the definition below. 

\begin{definition}
Let $U \subseteq V(G)$ and $V=V(G) \setminus U$. 
The \emph{upper subgraph} $G^U$ of $G$ w.r.t. $U$ is 
a spanning subgraph of $G$ with
$E(G^U)=E(G) \setminus E(G[V])$.
\end{definition}

In words,  $G^U$ is obtained from $G$ by removal of all the edges
whose both ends are outside of $U$. See Figure \ref{upgraph}
for an illustration of this notion.  

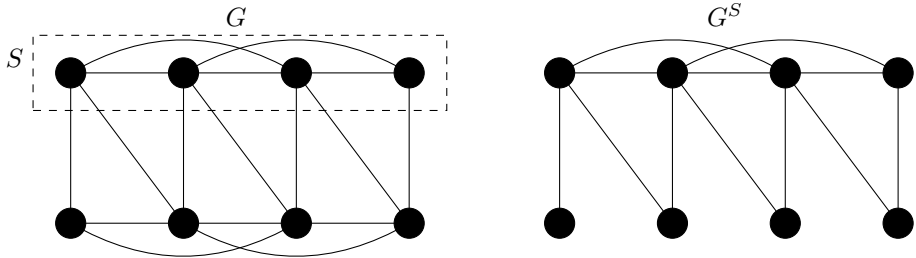
\begin{figure}[h]
\begin{tikzpicture}
\node [left]  at (0.5,4.2)  {$S$};
\node [above] at (3.2,4.5) {$G$};
\node [above] at (9.7,4.5) {$G^S$};
\draw [dashed] (0.5,4.5) rectangle (6,3.5);
\draw [fill=black]  (1,4) circle [radius=0.2];
\draw [fill=black]  (2.5,4) circle [radius=0.2];
\draw [fill=black]  (4,4) circle [radius=0.2];
\draw [fill=black]  (5.5,4) circle [radius=0.2];
\draw [fill=black]  (1,2) circle [radius=0.2];
\draw [fill=black]  (2.5,2) circle [radius=0.2];
\draw [fill=black]  (4,2) circle [radius=0.2];
\draw [fill=black]  (5.5,2) circle [radius=0.2];

\draw [fill=black]  (7.5,2) circle [radius=0.2];
\draw [fill=black]  (9,2) circle [radius=0.2];
\draw [fill=black]  (10.5,2) circle [radius=0.2];
\draw [fill=black]  (12,2) circle [radius=0.2];
\draw [fill=black]  (7.5,4) circle [radius=0.2];
\draw [fill=black]  (9,4) circle [radius=0.2];
\draw [fill=black]  (10.5,4) circle [radius=0.2];
\draw [fill=black]  (12,4) circle [radius=0.2];

\draw (1,2) to [out=330, in=210] (4,2); 
\draw (2.5,2) to [out=330, in=210] (5.5,2);
\draw (1,4) to [out=30, in=150] (4,4); 
\draw (2.5,4) to [out=30, in=150] (5.5,4);

\draw (1,4) --(2.5,4); 
\draw (2.5,4)--(4,4);
\draw (4,4)--(5.5,4); 
\draw (1,2) --(2.5,2); 
\draw (2.5,2)--(4,2);
\draw (4,2)--(5.5,2);
\draw (1,2)--(1,4);
\draw (2.5,2)--(2.5,4);
\draw (4,2)--(4,4);
\draw (5.5,2)--(5.5,4); 
\draw (1,4) --(2.5,2);
\draw (2.5,4) --(4,2);
\draw (4,4)--(5.5,2);

\draw (7.5,4) to [out=30, in=150] (10.5,4); 
\draw (9,4) to [out=30, in=150] (12,4);

\draw (7.5,4) --(9,4); 
\draw (9,4)--(10.5,4);
\draw (10.5,4)--(12,4); 
\draw (7.5,2)--(7.5,4);
\draw (9,2)--(9,4);
\draw (10.5,2)--(10.5,4);
\draw (12,2)--(12,4); 
\draw (7.5,4) --(9,2);
\draw (9,4) --(10.5,2);
\draw (10.5,4)--(12,2);

\end{tikzpicture}
\caption{An example of an upper subgraph}
\label{upgraph}
\end{figure}



\begin{definition} \label{defluimw}[\textsc{lu-mim width}]
Let $\pi=(v_1, \dots, v_n)$ be a permutation of $V(G)$ and 
denote $\{v_1, \dots, v_i\}$ by $V_i$.
Let $r_i$ be the size of the largest induced $(V_i, V(G) \setminus V_i)$-matching
of $G^{V_i}$. Let $r(\pi)=max_{i=1}^n r_i$.
The \emph{Linear Upper Induced Matching Width} (\luimw) of $G$
denoted by $lu(G)$ is the smallest $r(\pi)$ over all the permutations
$\pi$  of $V(G)$.  We call a permutation $\pi$ such that $r(\pi)=lu(G)$ a
\emph{witnessing permutation} for $lu(G)$. 
\end{definition}



\begin{example} \label{exam3}
In the graph $G$ of Figure \ref{upgraph}, consider the permutation $\pi$
first traversing all the top vertices from the left to the right and
then all the bottom vertices from the left to the right. 
Let $V_i$ be the set of all the top vertices (denoted by $S$ in the picture). 
It is not hard to see that the largest induced $(V_i,V(G) \setminus V_i)$- 
matching of $G^{V_i}$ is of size $1$. The widths of the rest of the prefixes
are also at most $1$, So, $r(\pi)=1$. Since the graph is connected, 
any permutation will have width at least one. So, we conclude that $lu(G)=1$.
\end{example}

The parameter \textsc{lu-mim width} can be considered as lying between
existing parameters \textsc{lmim width} and \textsc{lsim-width}.
In particular, to compute the width of a prefix $V_i$ for \textsc{lu-mim width}, the edges
having both ends in $V_i$ are discarded along with the edges having 
both of their ends out of $V_i$. In Example \ref{exam3}, with $V_i=S$, 
only the edges between the top and the bottom
vertices remain, so the largest size of an induced $(V_i,V(G) \setminus V_i)$-
matching of the resulting graph becomes $2$. For \textsc{lsim width}, no edges are discarded
at all,  so the width of $V_i$  is the largest induced $(V_i,V(G) \setminus V_i)$-matching 
for the whole $G$.

It is clear that for any graph $G$, its \textsc{lsim width} is smaller than or 
equal to its \textsc{lu-mim width} which, in turn, is smaller than or equal to
its \textsc{lmim width}. For the latter two we can, in fact, demonstrate 
a class of graphs where \textsc{lu-mim width} is bounded while \textsc{lmim width}
unbounded but we leave the exact relationship between the former two as an open question.
We postpone to Section \ref{sec:whynew} a more detailed discussion of relationship between the parameters
as well as justifying the need of the new parameter for bounding the size of \textsc{obdd}s.
The reason of this arrangement is that we need first to prove the main results of the 
paper so that we can refer to them for the purpose of the justification.

\section{OBDD bounds in terms of  LU-MIM width}  \label{sec:bounds}
In this section we establish upper and lower bounds on the size of \textsc{obdd}s
representing monotone two \textsc{cnf}s.
The upper bound is the more interesting of these two because it is based
on the following combinatorial statement.
 

\begin{theorem} \label{maincombstat}
Let $U \subseteq V(G)$ such that $V=V(G) \setminus U$ is independent.
Then $|{\bf TRACES}(U)|  \leq n^{r+1}$ where $r$ is the size of the largest induced 
$(U,V)$-matching.  
\end{theorem}

Before we provide a proof of Theorem \ref{maincombstat},
let us remark that if $U$ is independent (that is $G$ is a bipartite graph
with $U$ and $V$ being its parts) then the statement follows
from Sauer-Shelah lemma. 
This is just because, in this case, the size of the largest induced matching of $G$
is exactly the VC-dimension of ${\bf TRACES}(U)$. 
Indeed, let $W=\{w_1, \dots, w_q\}$ be a set of the largest size shattered by ${\bf TRACES}(U)$.
Then we can identify subsets $U_1,\dots, U_q$ such that $N(U_i) \cap W=\{w_i\}$
for $1 \leq i \leq q$. In particular, in each $U_i$ we can identify a vertex $u_i$ such
that $u_i$ is adjacent to $w_i$ but not adjacent to any other vertex of $W$.
Consequently, the edges $\{u_1,w_1\}, \dots, \{u_q,w_q\}\}$ constitute an induced matching.
Conversely, let $\{u_1,w_1\}, \dots, \{u_q,w_q\}$ be an induced matching. 
Then the set $\{w_1, \dots, w_q\}$ is shattered by neighborhoods of all possible
subsets of $\{u_1, \dots u_q\}$. Hence the VC dimension of ${\bf TRACES}(U)$ is at least $q$.

If $U$ is not an independent set, the first part of the above reasoning
does not work. Indeed, the vertices $u_1, \dots, u_q$ extracted from $U_1, \dots, U_q$
do not necessarily form an independent set and hence the resulting matching
is not necessarily induced. We were unable to upgrade the above argument 
to prove Theorem \ref{maincombstat} and hence we provide a self-contained proof.


\begin{proof}  (of Theorem \ref{maincombstat}.)

\begin{cclaim}  \label{indstep}
Let $S \subseteq U$ be an independent subset of $U$. 
Let $u \in S$. Suppose that $S \setminus \{u\}$ enables an induced
$(U,V)$-matching while $S$ does not. Then there is a subset  $S' \subset S$
enabling an induced $(U,V)$-matching such that  $N(S') \cap V=N(S) \cap V$.
\end{cclaim}

\begin{proof}
By induction on $|S|$. For $|S|=1$ the statement holds in a vacuous way. 
For each $u' \in S$ let $T(u')=(N(u') \cap V) \setminus (N(S \setminus \{u'\})  \cap V)$
be called the \emph{individual trace} of $u'$.  
Suppose all the individual traces are non-empty.
For all $u'$ fix an arbitrary $v' \in T(u')$.
Then $\{\{u',v'\}|u' \in S\}$ is an induced matching (recall that $V$ is independent)
contradicting our assumption.
It follows that there is $u' \in S$ such that $T(u')=\emptyset$.
But then $N(S) \cap V \subseteq  N(S \setminus \{u'\}) \cap V$ 
and hence $N(S) \cap V=N(S \setminus \{u'\}) \cap V$. 
If $S'$ enables an induced $(U,V)$-matching, we are done. 
Otherwise, apply the induction assumption to $S'$. $\square$
\end{proof}

\begin{cclaim} \label{indstep2}
Let $S \subseteq U$ be an independent subset of $U$. 
Then there is $S' \subseteq S$ enabling an induced $(U,V)$-matching 
such that $N(S) \cap V=N(S') \cap V$. 
\end{cclaim}

\begin{proof}
Let $q$ be the size of the largest subset of $S$ enabling 
an induced $(U,V)$-matching. We proceed by induction
on $|S|-q$. If it is zero then put $S'=S$.
Otherwise, let $S_0$ be a subset of $S$ of size $q$ enabling
an induced $(U,V)$ matching and let $u \in S \setminus S_0$. 
By Claim \ref{indstep}, there is $S_1 \subset S_0 \cup \{u\}$
enabling an induced $(U,V)$ matching such that
$N(S_1) \cap V=N(S_0 \cup \{u\}) \cap V$.

Let $S_2=S \setminus (S_0 \cup \{u\})$.
Then
\begin{multline} \label{smallmove}
N(S_1 \cup S_2) \cap V=(N(S_1) \cap V) \cup (N(S_2) \cap V)=\\
 (N(S_0 \cup \{u\}) \cap V) \cup (N(S_2) \cap V)=N(S_0 \cup \{u\} \cup S_2) \cap V=N(S) \cap V
\end{multline}

Further on, $S_1 \cup S_2$ has a subset of size at least $|S_1|$ enabling an induced $(U,V)$-matching.  
But $|S_2 \cup S_1|-|S_1|=|S_2|=|S|-q-1$.
Apply the induction assumption to $S_1 \cup S_2$ to find a subset $S_3 \subseteq S_1 \cup S_2$
enabling an induced $(U,V)$ matching such that $N(S_3) \cap V=N(S_1 \cup S_2) \cap V$.
Since $S_1 \cup S_2 \subseteq S$, $S_3 \subseteq S$ and $N(S_3) \cap V=N(S) \cap V$
by \eqref{smallmove}, we put $S'=S_3$.  $\square$
\end{proof} 

By assumption an independent subset of $U$ enabling an induced $(U,V)$ 
matching is of size at most $r$. 
It follows from Claim \ref{indstep2} that
${\bf TRACES}(U)=\{N(S) \cap V| S \in {\bf ISET}(U), |S| \leq r\}$.
Clearly the size of the right-hand set is upper bounded 
by the number of subsets of $U$ of size at most $r$ which is
clearly upper bounded as claimed in the theorem. $\blacksquare$
\end{proof}






\begin{theorem} \label{upperbound} [\textsc{obdd} bounds]
For $\varphi=\varphi(G)$,
$2^{lu(G)} \leq obdd(\varphi) \leq n^{O(lu(G))}$.
\end{theorem}

\begin{proof}
Let $\pi=(v_1, \dots, v_n)$ be a permutation of
$V(G)$ witnessing $lu(G)$. 
Let $V_i$ and $r_i$ be as in Definition \ref{defluimw}.
By combination of Lemma \ref{uppertraces} and Theorem \ref{maincombstat},
$|{\bf BF}(V_i)| \leq n^{r_i+1} \leq n^{lu(G)+1}$.
The upper bound follows from the second statement of Proposition \ref{propobddbounds}. 

For the lower bound we assume now that $\pi=(v_1, \dots, v_n)$ is an arbitrary
permutation with the meaning of $V_i$ and $r_i$ retained.
Furthermore, we assume that $(v_1, \dots, v_i)$ is selected so that
$r_i \geq lu(G)$ (such a prefix exists by definition of \textsc{lu-mim width}).  
We are going to show that $|{\bf TRACES}(V_i)| \geq 2^{r_i}$. 
The lower bound will then follow from combination of Lemma
\ref{uppertraces} and the first statement of Proposition \ref{propobddbounds}. 


Let $U^*=\{u_1, \dots, u_{r_i}\}$ 
be a subset of $V_i$ enabling an induced $(V_i, V(G) \setminus V_i)$-
matching of $G^{V_i}$ of size $r_i$ and let
$M=\{\{u_1,v_1\}, \dots, \{u_{r_i},v_{r_i}\}\}$ be the edges of this matching.
Let $U_1,U_2$ be two distinct subsets of $U^*$. We claim that 
$N(U_1) \cap (V(G) \setminus V_i) \neq N(U_2) \cap (V(G) \setminus V_i)$.
Indeed, assume w.l.o.g. that there is $u_j \in U_1 \setminus U_2$.
Then $v_j \in N(U_1) \cap (V(G) \setminus V_i)$ while $v_j \notin N(U_2) \cap (V(G) \setminus V_i)$.
Thus $2^{r_i}$ subsets of $U^*$ have pairwise distinct neighborhoods
in $V$ witnessing that $|{\bf TRACES}(V_i)| \geq 2^{r_i}$. 
$\blacksquare$

\end{proof}

\section{No tighter bounds} \label{sec:bestposs}
We are now going to prove that the bounds in the statement of  Theorem \ref{upperbound} are asymptotically best possible.
This will imply that the quasypolynmial gap between the upper and lower bounds cannot be narrowed down. 
For the lower bound the proof will be straightforward. 
For the upper bound we will need a 'gadgeted'  construction developed below. 

\begin{definition} \label{defskew}
Let $U=(u_1, \dots, u_q)$, $V=(v_1,\dots, v_q)$.
The graph $SKEW(U,V)$ 
over vertices $\{u_1, \dots, u_q,v_1, \dots, v_q\}$
has the set of edges $\{\{u_i,v_j\}i \leq j\}$. 

Let $U_1,\dots, U_p$ be mutually disjoint sequences
of $q$ elements. We define a graph $G$ over $U_1 \cup \dots \cup U_p$
(here we interpret $U_i$ as sets) as follows.
For each $1 \leq i \leq p-1$, $G[U_i \cup U_{i+1}]$
is $SKEW(U_i,U_{i+1})$.  We call $G$ a $p,q$-\emph{path of skewed graphs}.  
We call $U_1, \dots, U_p$ the \emph{sequence of layers}  of $G$ and give them numbers
$1, \dots, p$ in the order listed.  
\end{definition}


\begin{definition}
Let $P$ be a path.
The $1$-subdivision of $P$ is the graph 
obtained by introducing exactly one subdivision to each edge of $P$. 
\end{definition}

\begin{definition}
Let $G^1, \dots, G^r$  be $p,q$-paths of skewed graphs with respective
sequences $(U^1_1 \dots, U^1_p), \dots, (U^r_1, \dots, U^r_p)$ of layers.
Connect the vertices of each $U^1_i+\dots+U^r_i$ into a path $P$ in the order
specified and $1$-subdivide the resulting path. Let $G$ be the resulting graph.
We call $G$ a $p,q,r$-\emph{grid of skewed graphs} (we may omit the parameters
if they are not relevant in the context). 

The vertices $V(G^1) \cup \dots \cup V(G^r)$ are referred to as the
\emph{main vertices} and the vertices introduced by the $1$-subdivision
are the \emph{auxiliary vertices}. 
The subdivided paths are referred to as the \emph{layers} of $G$
with the $i$-th layer being the one containing $U^1_i, \dots, U^r_i$
as the main vertices.  Let us enumerate the main vertices of each layer
$i$ as in the sequence $U^1_i+ \dots,+U^r_i$ starting from $1$.
The number each vertex receives is the \emph{coordinate}
of this vertex.  
\end{definition}

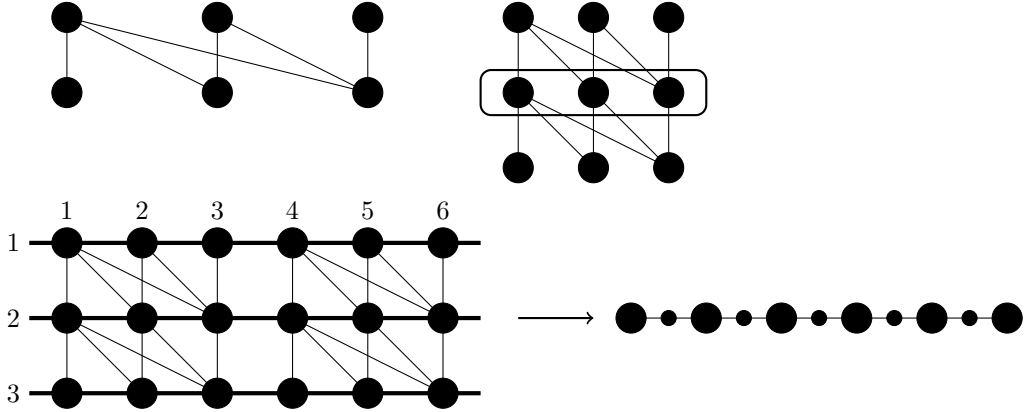
\begin{figure}[h]
\begin{tikzpicture}
\draw [fill=black]  (1,4) circle [radius=0.2];
\draw [fill=black]  (3,4) circle [radius=0.2];
\draw [fill=black]  (5,4) circle [radius=0.2];
\draw [fill=black]  (1,5) circle [radius=0.2];
\draw [fill=black]  (3,5) circle [radius=0.2];
\draw [fill=black]  (5,5) circle [radius=0.2];
\draw (1,4) --(1,5);
\draw (3,4) --(1,5);
\draw (3,4) --(3,5);
\draw (5,4) --(1,5);
\draw (5,4) --(3,5);
\draw (5,4) --(5,5);
\draw [fill=black]  (7,3) circle [radius=0.2];
\draw [fill=black]  (7,4) circle [radius=0.2];
\draw [fill=black]  (7,5) circle [radius=0.2];
\draw [fill=black]  (8,3) circle [radius=0.2];
\draw [fill=black]  (8,4) circle [radius=0.2];
\draw [fill=black]  (8,5) circle [radius=0.2];
\draw [fill=black]  (9,3) circle [radius=0.2];
\draw [fill=black]  (9,4) circle [radius=0.2];
\draw [fill=black]  (9,5) circle [radius=0.2];
\draw (7,3) --(7,4);
\draw (8,3) --(7,4);
\draw (8,3) --(8,4);
\draw (9,3) --(7,4);
\draw (9,3) --(8,4);
\draw (9,3) --(9,4);
\draw (7,4) --(7,5);
\draw (8,4) --(7,5);
\draw (8,4) --(8,5);
\draw (9,4) --(7,5);
\draw (9,4) --(8,5);
\draw (9,4) --(9,5);
\draw [thick, rounded corners] (6.5,4.3)  rectangle (9.5,3.7);

\draw [fill=black]  (1,0) circle [radius=0.2];
\draw [fill=black]  (1,1) circle [radius=0.2];
\draw [fill=black]  (1,2) circle [radius=0.2];
\draw [fill=black]  (2,0) circle [radius=0.2];
\draw [fill=black]  (2,1) circle [radius=0.2];
\draw [fill=black]  (2,2) circle [radius=0.2];
\draw [fill=black]  (3,0) circle [radius=0.2];
\draw [fill=black]  (3,1) circle [radius=0.2];
\draw [fill=black]  (3,2) circle [radius=0.2];
\draw (1,0) --(1,1);
\draw (2,0) --(1,1);
\draw (2,0) --(2,1);
\draw (3,0) --(1,1);
\draw (3,0) --(2,1);
\draw (3,0) --(3,1);
\draw (1,1) --(1,2);
\draw (2,1) --(1,2);
\draw (2,1) --(2,2);
\draw (3,1) --(1,2);
\draw (3,1) --(2,2);
\draw (3,1) --(3,2);

\draw [fill=black]  (4,0) circle [radius=0.2];
\draw [fill=black]  (4,1) circle [radius=0.2];
\draw [fill=black]  (4,2) circle [radius=0.2];
\draw [fill=black]  (5,0) circle [radius=0.2];
\draw [fill=black]  (5,1) circle [radius=0.2];
\draw [fill=black]  (5,2) circle [radius=0.2];
\draw [fill=black]  (6,0) circle [radius=0.2];
\draw [fill=black]  (6,1) circle [radius=0.2];
\draw [fill=black]  (6,2) circle [radius=0.2];
\draw (4,0) --(4,1);
\draw (5,0) --(4,1);
\draw (5,0) --(5,1);
\draw (6,0) --(4,1);
\draw (6,0) --(5,1);
\draw (6,0) --(6,1);
\draw (4,1) --(4,2);
\draw (5,1) --(4,2);
\draw (5,1) --(5,2);
\draw (6,1) --(4,2);
\draw (6,1) --(5,2);
\draw (6,1) --(6,2);
\draw [ultra thick] (0.5,0)--(6.5,0);
\draw [ultra thick] (0.5,1)--(6.5,1);
\draw [ultra thick] (0.5,2)--(6.5,2);
\draw [->, thick] (7,1)--(8,1);

\draw [fill=black]  (8.5,1) circle [radius=0.2];
\draw [fill=black]  (9,1) circle [radius=0.1];
\draw [fill=black]  (9.5,1) circle [radius=0.2];
\draw [fill=black]  (10,1) circle [radius=0.1];
\draw [fill=black]  (10.5,1) circle [radius=0.2];
\draw [fill=black]  (11,1) circle [radius=0.1];
\draw [fill=black]  (11.5,1) circle [radius=0.2];
\draw [fill=black]  (12,1) circle [radius=0.1];
\draw [fill=black]  (12.5,1) circle [radius=0.2];
\draw [fill=black]  (13,1) circle [radius=0.1];
\draw [fill=black]  (13.5,1) circle [radius=0.2];
\draw (8.5,1)--(13.5,1);
\node [left]  at (0.5,2)  {1};
\node [left]  at (0.5,1)  {2};
\node [left]  at (0.5,0)  {3};
\node [above]  at (1,2.2)  {1};
\node [above]  at (2,2.2)  {2};
\node [above]  at (3,2.2)  {3};
\node [above]  at (4,2.2)  {4};
\node [above]  at (5,2.2)  {5};
\node [above]  at (6,2.2)  {6};
\end{tikzpicture}
\caption{A grid of skewed graphs}
\label{graph1}
\end{figure}

Figure \ref{graph1} demonstrates a grid of skewed graphs. 
The top-left graph is $SKEW(U,V)$ where $U$ is the sequence 
of three vertices on the top enumerated from the left to the right
and $V$ is the respective sequence of the bottom vertices.
The graph on the top-right is a $3,3$-path of skewed graphs.
The graph has three layers enumerated from the top to the bottom.
The vertices of the second layer are surrounded by the oval.
The graph at the bottom-left is the $3,3,2$-grid of skewed graphs.
For the sake of a better visualization, the auxiliary vertices are not shown
and the layers are denoted by thick lines, the meaning of a thick 
line is specified on the bottom right of the picture. 
The grid has three layers and the coordinates of the main vertices
range from $1$ to $6$ as specified in the picture.

\begin{definition}  \label{defhor}
Let $G$ be an $p,q,r$-grid of skewed graphs. 
Let $U$ be a set of vertices one of each coordinate and 
none belonging to the last layer. For each vertex $u_1 \in U$
let $u_2$ be the vertex of the same coordinate lying at the next
layer. Let $V$ be the set of all vertices $u_2$.
Let $H$ be the subgraph of $G$ induced by $U \cup V$.
We call $H$ a \emph{horizontal} subgraph of $G$.
We call $U,V$ the \emph{top and bottom forming sets} of $H$. 
Let $M$ e the matching consisting of all the edges $\{u_1,u_2\}$
as above. We call $M$ the \emph{core matching} of $H$.
$U \cup V$ is partitioned into $r$ \emph{main intervals} $1, \dots, r$
where vertices of the $i$-th main interval are those having coordinates
$(i-1)*q+1, \dots i*q$.    
\end{definition}


\begin{lemma} \label{horizontfix}
With the notation as in Definition \ref{defhor}, 
both $|{\bf TRACES}_{H}(U)| \geq (q+1)^r$
and $|{\bf TRACES}_{H}(V)|\geq (q+1)^r$
\end{lemma}

\begin{proof}
We prove only the first statement, the second is symmetric.

Let $H_1, \dots, H_r$ be the subgraphs of $H$ induced
by the respective main intervals $1, \dots, r$.
For each $H_i$ denote $V(H_i) \cap U$ and $V(H_i) \cap V$
by $U_i$ and $V_i$, respectively.  
 
It is not hard to see that $H$ is the disjoint union 
of $H_1, \dots, H_r$, hence $|{\bf TRACES}_H(U)|=\prod_{i=1}^r |{\bf TRACES}_{H_i}(U_i)|$.
It is thus sufficient to prove that for each $i$, $|{\bf TRACES}_{H_i}(U_i)| \geq q+1$.
W.l.o.g. we only prove that $|{\bf TRACES}_{H_1}(U_1)| \geq q+1$.


For $U' \subseteq U_1$, let $first(U')$ be the vertex $u' \in U'$
located at the layer having the largest number (among the vertices of $U'$), 
and, among those vertices of $U'$ located at the layer, 
having the smallest coordinate. Let $u_1=first(U_1)$ and for
$2 \leq i \leq q$, $u_i=first(U_1 \setminus \{u_1, \dots, u_{i-1}\})$.
Let $v_1, \dots, v_q$ be  
the other ends of the edges of $M$  (the core matching of $H$) 
incident to $u_1, \dots, u_q$, respectively.
Observe that $H$ has no edge $\{u_i,v_j\}$ such that $i>j$.
Indeed, otherwise, either the layer of $u_j$ is smaller than the layer of $u_i$ or the coordinate
of $u_j$ is larger than the coordinate of $u_i$, both cases contradict the choice of vertices
$u_1, \dots, u_q$.  



Consider the sets $W_1, \dots W_{q+1}$ such that $W_{q+1}=\emptyset$
and $W_j=\{u_j\}$ for $ 1 \leq j \leq q$. 
It follows that for each $1  \leq j \leq q$, $v_j \in N(W_j) \cap V_1$ and
$w_j \notin N(W_k) \cap V_1$ for $k>j$.
It follows that the sets $N(W_1) \cap V_1, \dots, N(W_{q+1}) \cap V_1$
are all distinct thus confirming that  $|{\bf TRACES}_{H_1}(U_1)| \geq q+1$. 
$\blacksquare$
\end{proof}

\begin{lemma}  \label{horobddlower}
Let $G$ be a $p,q,r$-\emph{grid of skewed graphs} where 
$p>1,q>1,r \geq 1$, and $p=2*r \lceil \log q \rceil$. Let $n=V(G)$.  
Then for $\varphi=\varphi(G)$, $obdd(\varphi) \geq n^{r/2}$ 
for a fixed $r$ and a sufficiently large $n$.
\end{lemma}

\begin{proof}
We prove the $q^r$ lower bound instead of $n^{r/2}$.
Indeed $n^{r/2} \leq (2*q*p*r)^{r/2}=q^{r/2}* (4r^2 \lceil \log q \rceil)^{r/2} \leq q^r$
for a fixed $r$ and a sufficiently large $q$. We consider
an arbitrary permutation $\pi$ and show existence of a prefix
$\pi'$ such that $|{\bf TRACES}(\pi')| \geq q^r$. The lemma
will then follow from the combination of the first statement of 
Proposition \ref{propobddbounds} and Lemma \ref{uppertraces}.

Let $\pi'$ be the shortest prefix of $\pi$
containing all the vertices of some layer $x$. Assume 
existence of a layer $y$ none of which vertices are contained
in $\pi'$. Assume that $y>x$.  For each coordinate $i$, 
specify main vertices $u_i$, $v_i$
both having coordinate $i$ with the layer of $v_i$ being the
next after the layer of $u_i$ and such that $u_i \in \pi'$
while $v_i \notin \pi'$.   To see that such vertices exist,
start from the main vertex with coordinate $i$ at layer $x$ and iteratively
move down. Since the respective vertex of coordinate $i$ at $y$ is not in $\pi'$,
the required vertices $u_i,v_i$ will eventually be found.
The set $\{u_1, \dots, u_{qr},v_1, \dots, v_{qr}\}$ induce a horizontal subgraph
$H$ of $G$ with $U=\{u_1, \dots, u_{qr}\}$ being the top set.
Then $|{\bf TRACES}(\pi')| \geq |{\bf TRACES}_H(U)| \geq (q+1)^r$,
the last inequality follows from Lemma \ref{horizontfix}.

If $y<x$, the reasoning is symmetric and we use the second statement of Lemma 
\ref{horizontfix} rather than the first one. It remains to assume that
at least one vertex of each layer of $G$ is contained in $\pi'$.
Remove from $\pi'$ the last vertex and let $\pi^*$ be the resulting prefix.
By definition of $\pi'$,  in each layer of $G$ there is at least one vertex inside $\pi^*$
and at least one vertex outside $\pi^*$.
Since layers induce connected subgraphs of $G$, we can identify edges $\{u_i,v_i\}$
of $G$ for $1 \leq i \leq p$ with $u_i,v_i$ located at layer $i$, $u_i$ is contained in $\pi^*$
while $v_i$ is not.  We notice that the edges $\{u_i,v_i\}$ with odd indices form an induced
matching. Indeed, in $G$ two vertices are adjacent only if they are in the same
layer or in consecutive layers. Let $U$ be the set of vertices $u_i$ with $i$ being odd.
Applying the argument as in the lower bound proof for Theorem \ref{upperbound},
we observe that the neighborhoods of the subsets of $U$ in $V(G) \setminus \pi^*$
are pairwise distinct.
Taking into account the definition of $p$ and that 
$|U| = p/2$ by construction, 
we conclude that ${\bf TRACES}(\pi^*) \geq 2^{|U|} = 2^{p/2} \geq q^r$.
$\blacksquare$
\end{proof}

\begin{lemma} \label{lubound}
With the notation as in Lemma \ref{horobddlower}, $r \leq lu(G) \leq r+2$. 
\end{lemma}
\begin{proof} 
For the lower bound we argue as in Lemma \ref{horobddlower}.
Recall that for an arbitrary permutation $\pi$ we considered two cases. 
In the first case we observed existence of a prefix $\pi'$ such that there is a horizontal subgraph
$H$ of $G$ with all vertices of the top forming set contained in $\pi'$ and all vertices of the bottom
forming set being outside of $\pi'$ (or vice versa).  For each main interval take one edge of the core matching
whose vertex coordinates belong to this interval. These edges, taken together constitute an induced
matching of $G$ of size $r$.

If prefix $\pi'$ as above does not exist then there is a prefix $\pi^*$ such that for each
layer $1 \leq i \leq p$ there is an edge $\{u_i,v_i\}$ with $u_i \in \pi^*$ and $v_i \notin \pi^*$.
As we have observed the edges with odd indices comprise an induced matching of $G$
of size at least $r\log q>r$.

For the upper bound, we consider a 
permutation $\pi$ where vertices occur layer by layer: first layer $1$, then layer $2$ and so on.
Within each layer the vertices occur along the path induced by the layer starting from the 
main vertex with coordinate $1$. 

Consider a prefix $\pi'$  of $\pi$.  Let $x$ be the largest layer number (some of) whose vertices
are contained in $\pi'$. By definition of $\pi$ all the vertices whose layer numbers are smaller
than $x$ belong to $\pi'$.  It follows that the edges between $\pi'$ and $V(G) \setminus \pi'$
belong to one of the following categories.
\begin{enumerate}
\item Edges between layer $x$ and layer $x+1$. Suppose that $\pi'$ contains vertices of layer 
$x$ laying in intervals $1, \dots r'$. Then this category of edges can contribute at most $r'$
edges to an induced matching of $G^{\pi'}$.
\item Edges between layer $x-1$ and layer $x$. This category of edges can contribute at most $r-r'+1$
edges to the induced matching (the extra one is on the account that not all vertices of interval $r'$ and
layer $x$ may be present in $\pi'$)  so there may be an edge of vertices of interval $r'$ between layers
$x-1$ and $x$ contributing to the considered induced matching.
\item  Edges with both ends in layer $x$. Since $\pi'$ contains an initial fragment of the path of layer $x$,
there may be at most one such an edge.   
\end{enumerate}
Summing up the above three items, we conclude that the size of induced matching of $G^{\pi'}$
cannot be greater than $r+2$.


$\blacksquare$
\end{proof}

\begin{theorem} \label{bestpossible} [best possible bounds]
For every fixed $r \geq 1$ there are infinite  
classes $\mathcal{G}_1$ and $\mathcal{G}_2$
of graphs of \textsc{lu-mim width} $\Theta(r)$ and such
that for each $G_1 \in \mathcal{G}_1$, $obdd(\varphi(G_1)) \leq 2^{O(r)}$
while for each $G_2 \in \mathcal{G}_2$, $obdd(\varphi(G_2)) \geq n^{\Omega(r)}$. 
\end{theorem}

\begin{proof}
Let $\mathcal{G}_1$ be the set of all $p \times r$ grids for a sufficiently large $p$.
Each graph of this class has pathwidth of $\Theta(r)$ and hence the \textsc{obdd} size is at most $2^{O(r)}$ by 
\cite{VardiTWD}.  Because of the bounded degree, the pathwdith and the \textsc{lu-mim width}
of graphs in $\mathcal{G}_1$ are linearly related.  
Hence,  we conclude that for each $G_1 \in \mathcal{G}_1$, $obdd(\varphi(G_1))=2^{O(lu(G))}$. 
Let $\mathcal{G}_2$ be the class of $p,q,r$ grids for a sufficiently large $q$
and $p=2*r \lceil \log q \rceil$. The required properties are immediate from the combination of
Lemma \ref{horobddlower} and Lemma \ref{lubound}. $\blacksquare$ 
\end{proof}

\section{Why is the new parameter needed} \label{sec:whynew}
In this section we justify the need for the new parameter of \textsc{lu-mim width}.
In particular, we explain why we cannot use two existing closely related 
parameters: \textsc{lmim width} and \textsc{lsim width}. For the sake 
of completeness, let us define the latter two parameters.

\begin{definition}
Let $\pi=(v_1, \dots, v_n)$ be a permutation of $V(G)$. Denote $\{v_1, \dots, v_i\}$ by $V_i$.
Let $x_i$ be the largest size of an induced matching of $G[V_i,V(G) \setminus V_i]$ which is
the graph induced by the edges between $V_i$ and $V(G) \setminus V_i$.
Let $y_i$ be the largest size of an induced $(V_i,V(G) \setminus V_i)$-matching of $G$.
Let $x(\pi)$ be the maximum of all $x_i$ and let $y(\pi)$ be the maximum of all $y_i$. 
The Linear Maximum Induced Matching Width (\textsc{lmim width}) of $G$ denoted by $lmimw(G)$ is the minimum $x(\pi)$ over
all permutations $\pi$ of $V(G)$. 
The Linear Symmetric Induced Matching Width (\textsc{lsim width}) of $G$ denoted by $lsimw(G)$  is the minimum $y(\pi)$ over
all the permutations $\pi$ of $V(G)$. 
\end{definition}

The parameter \textsc{lmim width} cannot be used for our purposes
because it does not capture the lower bound for \textsc{obdd}s representing monotone $2$-\textsc{cnf}s.
In particular, below we demonstrate an infinite class of graphs having \textsc{lmim width}
of order of the square root of the number of vertices whose corresponding \textsc{cnf}s can
be represented by polynomial size \textsc{obdd}s.

\begin{figure}[h]
\begin{tikzpicture}
\draw [fill=black]  (1,1) circle [radius=0.2];
\draw [fill=black]  (3,1) circle [radius=0.2];
\draw [fill=black]  (5,1) circle [radius=0.05];
\draw [fill=black]  (5.5,1) circle [radius=0.05];
\draw [fill=black]  (6,1) circle [radius=0.05];
\draw [fill=black]  (8,1) circle [radius=0.2];
\draw [fill=black]  (10,1) circle [radius=0.2];
\draw [fill=black]  (1,2) circle [radius=0.2];
\draw [fill=black]  (3,2) circle [radius=0.2];
\draw [fill=black]  (5,2) circle [radius=0.05];
\draw [fill=black]  (5.5,2) circle [radius=0.05];
\draw [fill=black]  (6,2) circle [radius=0.05];
\draw [fill=black]  (8,2) circle [radius=0.2];
\draw [fill=black]  (10,2) circle [radius=0.2];

\draw [fill=black]  (1,4) circle [radius=0.2];
\draw [fill=black]  (3,4) circle [radius=0.2];
\draw [fill=black]  (5,4) circle [radius=0.05];
\draw [fill=black]  (5.5,4) circle [radius=0.05];
\draw [fill=black]  (6,4) circle [radius=0.05];
\draw [fill=black]  (8,4) circle [radius=0.2];
\draw [fill=black]  (10,4) circle [radius=0.2];
\draw [fill=black]  (1,5) circle [radius=0.2];
\draw [fill=black]  (3,5) circle [radius=0.2];
\draw [fill=black]  (5,5) circle [radius=0.05];
\draw [fill=black]  (5.5,5) circle [radius=0.05];
\draw [fill=black]  (6,5) circle [radius=0.05];
\draw [fill=black]  (8,5) circle [radius=0.2];
\draw [fill=black]  (10,5) circle [radius=0.2];

\draw (1,1) --(1,2);
\draw (3,1) --(3,2);
\draw (8,1) --(8,2);
\draw (10,1) --(10,2);
\draw (1,4) --(1,5);
\draw (3,4) --(3,5);
\draw (8,4) --(8,5);
\draw (10,4) --(10,5);

\draw [fill=black]  (4.5,3) circle [radius=0.05];
\draw [fill=black]  (5,3) circle [radius=0.05];
\draw [fill=black]  (5.5,3) circle [radius=0.05];
\draw [fill=black]  (6,3) circle [radius=0.05];
\draw [fill=black]  (6.5,3) circle [radius=0.05];

\draw (1,2) --(1,2.8);
\draw (3,2) --(3,2.8);
\draw (8,2) --(8,2.8);
\draw (10,2) --(10,2.8);
\draw (1,4) --(1,3.2);
\draw (3,4) --(3,3.2);
\draw (8,4) --(8,3.2);
\draw (10,4) --(10,3.2);

\draw [dashed,rounded corners] (0.7,1.3) rectangle (10.3,0.7);
\draw [dashed,rounded corners] (0.7,2.3) rectangle (10.3,1.7);
\draw [dashed,rounded corners] (0.7,4.3) rectangle (10.3,3.7);
\draw [dashed,rounded corners] (0.7,5.3) rectangle (10.3,4.7);

\draw [->] (11,3) --(10.4,1);
\draw [->] (11,3) --(10.4,2);
\draw [->] (11,3) --(10.4,4);
\draw [->] (11,3) --(10.4,5);

\node [right] at (11,3) {CLIQUES};
\end{tikzpicture}
\caption{Schematic illustration of graphs $H_n$}
\label{graph3}
\end{figure}
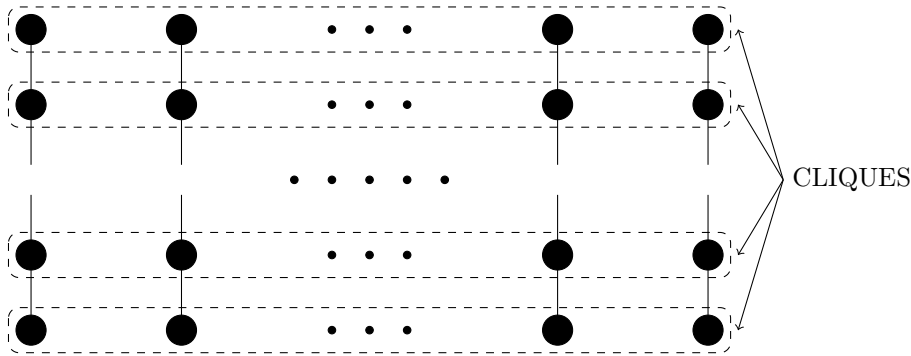

\begin{theorem}
For each integer $r \geq 2$, there is an infinite class of graphs $H_r$ of $n=r^2$ vertices 
such that $lu(H_r)=2$ (and hence $\varphi(H_r)$ can be represented
by an \textsc{obdd} of size at most $n^{O(1)}$ by Theorem \ref{upperbound} ) while $lmimw(H_r) \geq (r-1)/2$.
\end{theorem}

{\bf Proof.}
$V(H_r)$ consists of disjoint union of sets $V_1, \dots, V_r$
of $r$ vertices each. Each $V_i$ is a clique in $H_r$. Denote the vertices
of each $V_i$ by $v_{i,1}, \dots, v_{i,r}$.
The graph $H_r$ has paths $v_{1,i}, \dots, v_{r,i}$ for each $1 \leq i \leq r$. 
$E(H_r)$ contains no other edges besides those specified above.
Figure \ref{graph3} schematically illustrates the graphs $H_r$. 

To demonstrate that $lu(H_r)$  is small,
consider the permutation of $V(H_r)$ by the alphabetic ordering of
their indices, that is, $v_{1,1}, \dots, v_{1,r}, \dots, v_{r,1}, \dots, v_{r,r}$.
Let $V'$ be the set of vertices of a prefix of this permutation.
Let $1 \leq q \leq r$ be such that $V_q \cap V' \neq \emptyset$ while
for each $q<i \leq r$, $V_i \cap V'=\emptyset$. 
It follows that for each $1 \leq i <q$, $V_i \subseteq V'$.
Hence, by construction, any edge connecting $V'$ to $V(H_r) \setminus V'$
has an end either in $V_{q}$ or in $V_{q-1}$. Thus for any three edges between
$V'$ and $V(H_r) \setminus V'$ either two of them have an end in $V_q$
or two of them have an end in $V_{q-1}$. In both cases these ends are connected
by an edge with both ends lying in $V'$ and hence such edges cannot constitute
an induced matching of $H_r^{V'}$. We conclude that the largest possible
size of the such an induced matching is $2$.
It follows from Theorem \ref{upperbound} that $\varphi(H_r)$ can be represented 
by an \textsc{obdd} of size upper bounded by $n^{O(1)}$. 

Let us know establish an $\Omega(r)=\Omega(n^{1/2})$ lower bound on $lmimw(H_r)$. 
For vertices $v_{i,j}$ of $H_r$ let us call their first coordinates \emph{rows}
and their second coordinates \emph{columns}. 
Let $\pi$ be an arbitrary permutation of $V(H_r)$.
Let $\pi_0$ be the longest prefix of $\pi$ that does not contain
vertices with all the row coordinates and does not contain vertices
with all the column coordinates. Since this is already not the case for the 
immediate successor of $\pi_0$, it is either that $\pi_0$ contains vertices 
with $r-1$ row coordinates or vertices with $r-1$ column coordinates.
Assume the former. 
Then there is a set $I$ of $r-1$ rows $i$ such that $\pi_0$ contains some $v_{i,j}$. 
By assumption, for each $i \in I$, there is some $q$ such that $v_{i,q} \notin \pi_0$.
Since each $V_i$ is connected we can identify, for each $i \in I$ an edge $\{v_{i,j_1}, v_{i,j_2}\}$
such that one end of this edge is in $\pi_0$ while the other end is outside. 
At least half of such edges have the same parity of the row. Let $M$ be a such a subset of edges.
By definition of $H_r$ vertices with the same row parity are not adjacent hence this
matching is induced in $H_r$ and of size at least $(r-1)/2$ by definition.

It remains to assume that there is a set $I$ of $r-1$ columns $j$ such that there is a vertex $v_{i,j} \in \pi_0$.
By assumption, at least one vertex of $v_{1,j}, \dots, v_{r,j}$ does not belong to $\pi_0$.
As $v_{1,j}, \dots, v_{r,j}$ induce a path, there is an edge $\{v_{i,j},v_{i+1,j}\}$ such
that one end belong to $\pi_0$ while the other end is outside.  Let $E'$ be a set of such edges
one per column of $I$. 
For an edge $\{v_{i,j},v_{i+1,j}\}$  of $E'$ we call its end that belongs to $\pi_0$
the \emph{inner end} and the other one the \emph{outer end}. 
We call the edge \emph{even} if the row of the inner end is even and \emph{odd}
otherwise. Clearly at least $(r-1)/2$ edges of $E'$ have the same parity.
Assume without loss of generality that these are even edges. 
It remains to demonstrate that there are no distinct columns $j_1$ and $j_2$.
such that the inner end of the edge of $E'$ of column $j_1$ is adjacent to the outer
end of the edge corresponding to $j_2$. Since the columns are different the adjacency
may be only because the adjacent ends belong to the same clique $V_i$. 
But this means that the row of the inner end of $j_2$ is odd, a contradiction. 

$\blacksquare$ 


Regarding \textsc{lsim width}, we do not know whether an appropriate upper bound
can be captured using this parameter.
A straightforward attempt to do so would be to 
try to prove Theorem \ref{maincombstat}
without the assumption that $V=V(G) \setminus U$ is
independent (if this worked then \textsc{lsim width} can
simply replace \textsc{lu-mim width} in the proof of Theorem \ref{upperbound}). 
However, such an attempt fails.
Consider a situation where $U$ is independent, $V$ is a clique,
there is a perfect matching $M$ between $U$ and $V$, and no other
edges are present in $G$.  In this case $|{\bf TRACES}(U)|=2^{|U|}$ while
the size of the largest induced $(U,V)$-matching is $1$. 


We believe that the right approach towards understanding the role 
of \textsc{lsim-width} in the considered context is to first clarify 
the relationship of this parameter with \textsc{lu-mim width}.
We conjecture that there is
an infinite class of graphs where the former is
bounded from above by a constant while the latter is bounded from below by $\Omega(n^{\alpha})$
for some $\alpha>0$. 
If this is the case then it can be concluded that \textsc{lsim width} cannot be used
for capturing upper bounds for \textsc{obdd}s representing monotone $2$-\textsc{cnf}s. .

\section{Future research} \label{sec:futres}
In this section we discuss several interesting open questions 
related to representation of monotone $2$-\textsc{cnf}s by 
circuit models more powerful than \textsc{obdd}s.
A natural question in this direction is to consider Nondeterministic Read-Once Branching Programs 
($1$-\textsc{nbp}s) instead of \textsc{obdd}s,
For example, is it true that the size of $1$-\textsc{nbp} representing a monotone $2$-\textsc{cnf}
$\varphi=\varphi(G)$ is lower bounded by $2^{\Omega(lu(G))}$?  

Similarly to \textsc{mim width} and \textsc{sim width},
it is possible to formulate the 'non-linear'  version of \textsc{lu-mim width}
in terms of the branch decompositions rather than permutations. 
It is interesting to investigate whether the size of Decomposable Negation Normal Forms (\textsc{dnnf}s) 
or its restricted classes such as Deterministic \textsc{dnnf}s representing monotone $2$-\textsc{cnf}s can be
captured by this non-linear parameter.  
We conjecture that 
the resulting non-linear parameter 'captures' the size of Structural Deterministic 
\textsc{dnnf}s but for more general models the situation is unclear and is likely
to depend on the situation with $1$-\textsc{nbp}s.
Our belief relies on a plausible analogy with the bounded degree case
where pathwidth captures the size of $1$-\textsc{nbp}s while its non-linear
counterpart (that is treewidth) captures the size of \textsc{dnnf}s \cite{AmarilliTCS}.

Finally, a natural question arising from the results
of this paper is to capture the size of \textsc{obdd}s
on monotone \textsc{cnf}s of higher arity.
One possibility to achieve this might be through a concise
generalization of \textsc{lu-mim width} to hypergraphs.






\end{document}